\documentclass[graybox]{SNmult}

\usepackage{type1cm}        
\usepackage{makeidx}         
\usepackage{graphicx}        
\usepackage{tikz}
\usetikzlibrary{arrows.meta}
\usepackage{subcaption}

\usepackage{multicol}        
\usepackage[bottom]{footmisc}
\usepackage{mathtools}

\usepackage{newtxtext}       %
\usepackage[varvw]{newtxmath}       

\makeindex             

\begin{document}
\title*{Isogeometric Discretizations for the Spectrum of the Laplace Operator: Outlier-Free Spline Bases}
\titlerunning{Isogeometric Discretizations: Outlier-Free Spline Bases}
\author{Damiano Ricci\orcidID{0009-0007-2239-0299}
and\\ Carla Manni\orcidID{0000-0002-1519-4106}
and\\ Hendrik Speleers\orcidID{0000-0003-4110-3308}
}

\authorrunning{Damiano Ricci, Carla Manni, Hendrik Speleers}
\institute{Damiano Ricci, Carla Manni, Hendrik Speleers \at University of Rome Tor Vergata, Department of Mathematics, Via della Ricerca Scientifica 1, 00133 Rome, Italy,
\email{damianoricci1998@gmail.com}, \email{manni@mat.uniroma2.it}, \email{speleers@mat.uniroma2.it}
}
%
%
\maketitle
{\vspace*{-1cm}}

\abstract{
Optimal spline subspaces are an elegant and efficient tool to remove spurious outliers in isogeometric Galerkin discretizations for the approximation of the spectrum of the Laplace operator. For practical purposes, it is valuable to have a basis construction for such spaces with good computational and spectral properties. We provide a characterization of the bases that enjoy a B-spline-like support structure and whose mass and stiffness matrices are simultaneously diagonalizable. It turns out that these mass and stiffness matrices admit explicitly known closed-form expressions for their eigenvalues, implying that the considered bases are outlier-free. A numerical procedure for constructing such bases is also presented.
\keywords{Isogeometric analysis $\cdot$ Eigenvalue problems $\cdot$ Outlier-free discretizations $\cdot$ Optimal spline subspaces}
}

\section{Introduction}
\label{sec:1}
Isogeometric Galerkin methods based on spline spaces with maximal smoothness on uniform grids provide an excellent approximation of almost the entire spectrum of the Laplace operator, in contrast to what is observed when $C^0$ finite element methods are used \cite{Cottrell2006}. These spline discretizations, however, suffer from a small number of spurious eigenvalues, called outliers. The number of outliers increases with the polynomial degree but, in the univariate case, is independent of the number of degrees of freedom.

Although limited in number, outliers degrade the overall quality of the discretization, affect the choice of time steps in explicit dynamics, and jeopardize the accurate approximation of high frequencies. It is well known that the presence of outliers is related to the treatment of boundary conditions, which induces low-rank perturbations in the otherwise clean algebraic structure of the matrices involved in the discretization \cite{Garoni2014,Manni2022}.

Recently, spline subspaces have been developed that are theoretically proven to be \emph{outlier-free}, while still preserving the good approximation accuracy of the entire spectrum. They are obtained by imposition of special boundary conditions and are optimal in the sense of Kolmogorov $n$-widths \cite{Manni2022}.

The efficient use of such outlier-free spaces relies on the ability to represent their elements in terms of a basis with good computational properties. Since the spaces of interest are subspaces of classical spline spaces of maximal smoothness whose elements satisfy additional homogeneous boundary conditions, it is natural to seek bases that coincide with B-splines, except for a few basis elements affected by the new boundary conditions. For any degree $p$ and for the most common boundary conditions (Dirichlet, Neumann, and mixed), a basis of the corresponding optimal space with this structure has been proposed in \cite{Divona2019,Manni2022}, which can be easily constructed in terms of cardinal B-splines; see also \cite{FloaterSande2019,Lamsahel2025,Takacs2016} for some related work.

It turns out that the bases considered in \cite{Divona2019,Manni2022} additionally possess remarkable spectral properties: the eigenvectors of the corresponding stiffness matrix all coincide with the eigenvectors of the corresponding mass matrix, i.e., these matrices are simultaneously diagonalizable. Moreover, the corresponding mass and stiffness matrices admit explicitly known closed-form expressions for their eigenvalues. Therefore, such bases are called \emph{outlier-free}.
The above properties follow from the fact that the considered matrices exhibit a Toeplitz-minus-Hankel or Toeplitz-plus-Hankel structure \cite{Lamsahel2025}; see also \cite{Deng2021}. This spectral knowledge is relevant not only from a theoretical point of view but also because it enables the construction of efficient solvers for the associated linear systems.

In this paper, we characterize all bases of optimal spline spaces -- for any degree $p$ and any type of boundary condition (Dirichlet, Neumann, and mixed) -- that satisfy the following properties (see Definition~\ref{DEFINIZIONE}):
\begin{itemize}
\item the basis elements have the same support structure as those of the bases in \cite{Divona2019,Manni2022};
\item the corresponding mass and stiffness matrices are simultaneously diagonalizable.
\end{itemize}
We prove that, for each optimal space, any basis enjoying these properties can be obtained from the corresponding basis in \cite{Divona2019,Manni2022} by  an orthogonal matrix; see Theorem~\ref{TEOREMA}. This implies that the corresponding mass and stiffness matrices share the same eigenvalues; see Corollary~\ref{Basi_outlierfree}. In other words, all the selected bases are outlier-free.

The remainder of the paper is divided into five sections. Section~\ref{sec:2} collects notation and preliminary material on splines and graphs, while Section~\ref{sec:3} briefly summarizes the state of the art on outlier-free optimal spline spaces and their bases. Section~\ref{sec:4} contains the main result of the paper (Theorem~\ref{TEOREMA}) and its proof. Section~\ref{sec:5} presents a numerical optimization procedure that can be used to construct outlier-free bases with no pre-knowledge about any of them. We end in Section~\ref{sec:6} with some final remarks.

\section{Preliminaries and Notation}\label{sec:2}
In this section, we introduce some notation and basic results related to splines and graphs, of interest in the rest of the paper.

\subsection{Splines of Maximal Smoothness}
We start by summarizing some results about spline spaces and their bases. For more details, we refer to \cite{Lyche2018} and references therein.

Let $1\leq p \in \mathbb{N}$ and let
$\boldsymbol{\tau} \coloneqq \{0 = \tau_1 < \cdots < \tau_{m+1} = 1\} $
be a partition of the interval $[0,1]$ in $m$ elements.
The \emph{spline space} of degree $p$ and maximal smoothness on the partition $\boldsymbol{\tau}$ is defined by
\begin{equation*}
\mathbb{S}_{p,\boldsymbol{\tau}}
\coloneqq
\left\{
s \in C^{p-1}([0,1]) :
s_{|[\tau_i,\tau_{i+1}]} \in \mathbb{P}_p,\ i=1,\ldots,m
\right\},
\end{equation*}
where $\mathbb{P}_p$ is the space of polynomials of degree at most $p$.
This space has dimension $m+p$.

Let $N \in \mathbb{N}$ such that $N > p \geq 1$, and let $\boldsymbol{\xi}$ be a non-decreasing sequence of real values, called \emph{knots},
\begin{equation*}
\boldsymbol{\xi} \coloneqq \{\xi_1 \le \xi_2 \le \cdots \le \xi_{N+p+1}\}.
\end{equation*}
For $j=1,\ldots,N$, the $j$-th \emph{B-spline} of degree $p$
associated with the knots $\boldsymbol{\xi}$ is defined recursively by
\begin{equation*}
\mathcal{B}_{j,p,\boldsymbol{\xi}}(x)
\coloneqq \frac{x - \xi_j}{\xi_{j+p} - \xi_j} \, \mathcal{B}_{j,p-1,\boldsymbol{\xi}}(x)
\;+\;
\frac{\xi_{j+p+1} - x}{\xi_{j+p+1} - \xi_{j+1}} \,
\mathcal{B}_{j+1,p-1,\boldsymbol{\xi}}(x),
\end{equation*}
with
\begin{equation*}
\mathcal{B}_{j,0,\boldsymbol{\xi}} \coloneqq
\begin{cases}
1, & \text{if } x \in [\xi_j, \xi_{j+1}), \\
0, & \text{otherwise}.
\end{cases}
\end{equation*}
Here fractions with zero denominator are set equal to zero and we adopt the convention
\begin{equation*}
\mathcal{B}_{j,p,\boldsymbol{\xi}}(\xi_{N+p+1})\coloneqq\lim_{x \to \xi_{N+p+1}^{-}}\mathcal{B}_{j,p,\boldsymbol{\xi}}(x).
\end{equation*}
The $N=m+p$ B-splines of degree $p$ associated with the knots
\begin{equation}\label{osaitu}
\{\underbrace{\tau_1 = \cdots = \tau_1}_{p+1}
< \tau_2 < \cdots <\tau_m<
\underbrace{\tau_{m+1} = \cdots = \tau_{m+1}}_{p+1}\}
\end{equation}
form a basis of $\mathbb{S}_{p,\boldsymbol{\tau}}$, which is denoted by $\mathscr{B}_{p,\boldsymbol{\tau}}$ and referred to as the \emph{open-knot B-spline basis}.

Finally, the \emph{cardinal B-spline} of degree $p$ is defined recursively by
\begin{equation*}
\mathcal{N}_p(x)
\coloneqq \frac{x}{p}\,\mathcal{N}_{p-1}(x)
   + \frac{p+1-x}{p}\,\mathcal{N}_{p-1}(x-1),
\end{equation*}
with
\begin{equation*}
\mathcal{N}_0(x) \coloneqq
\begin{cases}
1, & \text{if } x \in [0,1),\\
0, & \text{otherwise}.
\end{cases}
\end{equation*}
For $t \in \mathbb{Q}$, we set
$\mathcal{N}_{t,p} \coloneqq \mathcal{N}_p(x - t)$, i.e.,
a rational translate of the cardinal B-spline of degree $p$.

\subsection{Graph Theory} \label{grafi}
We now collect some basic facts on graphs, to be used later in the paper. For more details, we refer to \cite{GodsilRoyle2001} and references therein.

An \emph{undirected graph} is a pair $G \coloneqq (V,E)$, where
$V \coloneqq \{1,\ldots,n\}$ is the set of \emph{nodes} (or vertices) and
$E \subseteq \big\{\{u,v\}\subseteq V: u\neq v\big\}$ is the set of
\emph{edges} (unordered pairs of nodes).
We assume graphs to be simple (no loops and no multiple edges).

With a given graph, one can associate a matrix encoding the node--edge relations.
Let $G \coloneqq (V,E)$ be an undirected graph with $|V|=n$ and $|E|=m$.
After fixing an arbitrary orientation for each edge $e\in E$
(which is only used to define consistent signs),
the \emph{incidence matrix} is the matrix
$B\in\mathbb{R}^{n\times m}$, whose columns are indexed by edges
and rows by nodes, defined by
\begin{equation*}
B_{v,e} \coloneqq
\begin{cases}
-1, & \text{if $v$ is the \emph{head} of $e$},\\
+1, & \text{if $v$ is the \emph{tail} of $e$},\\
 0, & \text{otherwise}.
\end{cases}
\end{equation*}

An undirected graph $G$ is said to be \emph{connected} if for every
pair of nodes $u,v\in V$ there exists a path joining them.
In general, $G$ can be decomposed into $c(G)$ maximal connected components, which are pairwise disjoint on the node set.

The following classical theorem characterizes the rank of the incidence matrix in terms of the connected components of the graph.
\begin{theorem}\label{Grafi1}
Let $G$ be an undirected graph with $n$ nodes and $c(G)$ maximal connected components. Then,
\begin{equation*}
\mathrm{rank}(B) = n - c(G),
\end{equation*}
where $B$ is an incidence matrix (for any orientation of the edges). In particular,	$G$  is connected if and only if $\mathrm{rank}(B) = n-1$.
\end{theorem}

\section{Laplace Eigenvalue Problem: Isogeometric Discretizations}\label{sec:3}
We consider the eigenvalue problem associated with the 1D Laplace operator,
\begin{equation}\label{Laplace}
-u''= \omega^2 u, \quad \text{in } (0,1),
\end{equation}
and the following standard boundary conditions:
\begin{itemize}
\item Dirichlet boundary conditions (also referred to as fixed or type 0 boundary conditions),
\begin{equation}\label{Dirichlet}
u(0)=u(1)=0;
\end{equation}
\item Neumann boundary conditions (also referred to as free or natural or type 1 boundary conditions),
\begin{equation}\label{Neumann}
u'(0)=u'(1)=0;
\end{equation}
\item a combination of the previous ones (also referred to as mixed or type 2 boundary conditions),
\begin{equation}\label{Miste}
u(0)=u'(1)=0.
\end{equation}
\end{itemize}
The non-trivial exact solutions of \eqref{Laplace} subject to one of the boundary conditions \eqref{Dirichlet}--\eqref{Miste} form a numerable set of trigonometric functions, respectively,
\begin{alignat*}{3}
u_l(x) &\coloneqq \sin(\omega_l x),\quad &\omega_l &\coloneqq l \pi, \quad &l &= 1,2,\ldots \\
u_l(x) &\coloneqq \cos(\omega_l x),\quad &\omega_l &\coloneqq l \pi, \quad &l &= 0,1,2,\ldots \\
u_l(x) &\coloneqq \sin(\omega_l x),\quad &\omega_l &\coloneqq (l-1/2)\pi, \quad &l &=1,2,\ldots
\end{alignat*}

The weak form of problem \eqref{Laplace} reads as follows. Find non-trivial $u \in \mathbb{V}$ and $\omega^2 \in \mathbb{R}$ such that
\begin{equation*}
\int_0^1 u'(x)\, v'(x)\, \textrm{d}x
= \omega^2 \int_0^1 u(x)\, v(x)\, \textrm{d}x
\quad \forall v \in \mathbb{V},
\end{equation*}
where $\mathbb{V}$ is selected in accordance with the chosen boundary conditions:
\begin{itemize}
\item Dirichlet boundary conditions: $\mathbb{V}=\{v \in \mathbb{H}^1(0,1) : v(0) = v(1) = 0 \}$;
\item Neumann boundary conditions: $\mathbb{V}=\mathbb{H}^1(0,1)$;
\item mixed boundary conditions: $\mathbb{V}=\{v \in \mathbb{H}^1(0,1) : v(0) = 0 \}$.
\end{itemize}

Following the Galerkin approach, we consider a finite-dimensional subspace $\mathbb{V}_h$ of $\mathbb{V}$ spanned by a basis $\{\varphi_1, \ldots, \varphi_{n_h} \}$ and we look for approximate values $\omega_h$ to $\omega$ by solving
\begin{equation*}
K_h \mathbf{u}_h = (\omega_h)^2M_h\mathbf{u}_h,
\end{equation*}
where the stiffness matrix $K_h$ and the mass matrix $M_h$ consist of the elements
\begin{equation}\label{stiff-mass}
K_{h,i,j} \coloneqq \int_0^1 \varphi_j'(x)\varphi_i'(x)\, \textrm{d}x, \quad M_{h,i,j} \coloneqq \int_0^1\varphi_j(x)\varphi_i(x)\, \textrm{d}x,
\end{equation}
with $i,j = 1,\ldots,n_h$.
For $l = 1, \ldots,n_h$, an approximation of the frequency $\omega_l$ is given by the square root of the $l$-th eigenvalue of $M_h^{-1}K_h$, denoted by $\omega_{h,l}$. Here we assume that those eigenvalues are given in ascending order. Similarly, an approximation of the eigenfunction $u_l$ is obtained by considering
\begin{equation*}
u_{h,l}(x) \coloneqq \sum_{i=1}^{n_h} u_{h,l,i}\varphi_i(x),
\end{equation*}
where $\mathbf{u}_{h,l} \coloneqq (u_{h,l,1}, \ldots, u_{h,l,n_h})$ is the $l$-th eigenvector of $M_h^{-1}K_h$, properly normalized.
More information on this eigenvalue problem can be found in \cite{Boffi2010}.

In the classical isogeometric approach, the finite-dimensional subspace $\mathbb{V}_h$ is selected as a proper subspace of the spline space $\mathbb{S}_{p,\boldsymbol{\tau}}$ in accordance with the chosen boundary conditions:
\begin{itemize}
\item Dirichlet boundary conditions: $\mathbb{V}_h=\{s \in \mathbb{S}_{p,\boldsymbol{\tau}} : s(0) = s(1) = 0 \}$;
\item Neumann boundary conditions: $\mathbb{V}_h= \mathbb{S}_{p,\boldsymbol{\tau}}$;
\item mixed boundary conditions: $\mathbb{V}_h=\{s \in \mathbb{S}_{p,\boldsymbol{\tau}} : s(0) =  0 \}$.
\end{itemize}
The above choices, with a uniform partition $\boldsymbol{\tau}$, produce a very good approximation of the continuous spectrum, with increasing accuracy as the degree $p$ increases \cite{Cottrell2006}. However, the scheme still suffers from some numerical artifacts: a small portion of frequencies is poorly approximated. These spurious approximations are called \emph{outliers}; see \cite{Manni2022} and references therein.

\subsection{Outlier-Free Spline Spaces}
The outliers in the isogeometric eigenvalue approximation of the Laplace operator can be removed by taking special spline subspaces as the discretization space $\mathbb{V}_h$. These subspaces turn out to be optimal in the sense of Kolmogorov $n$-widths for appropriate function classes identified by problem \eqref{Laplace} with boundary conditions \eqref{Dirichlet}--\eqref{Miste}; see \cite{FloaterSande2019,Manni2022}. More precisely, let us consider the following $n$-dimensional spline spaces:
\begin{itemize}
\item Dirichlet boundary conditions:
\begin{equation}\label{S_p_0}
\mathbb{S}_{p,0} \coloneqq \bigl\{ s \in \mathbb{S}_{p,\boldsymbol{\tau}_{p,0}} : \partial^\alpha s(0) = \partial^\alpha s(1) = 0, \  0 \le \alpha \le p, \ \alpha \ \text{even} \bigr\},
\end{equation}
where
\begin{equation*}
\boldsymbol{\tau}_{p,0} \coloneqq
\begin{cases}
\,\bigl\{0, \frac{1}{n+1}, \frac{2}{n+1}, \ldots, \frac{n}{n+1}, 1\bigr\}, & p \text{ odd},\\[3pt]
\,\bigl\{0, \frac{1/2}{n+1}, \frac{3/2}{n+1}, \ldots, \frac{n+1/2}{n+1}, 1\bigr\}, & p \text{ even};
\end{cases}
\end{equation*}
\item Neumann boundary conditions:
\begin{equation}\label{S_p_1}
\mathbb{S}_{p,1} \coloneqq \bigl\{ s \in \mathbb{S}_{p,\boldsymbol{\tau}_{p,1}} : \partial^\alpha s(0) = \partial^\alpha s(1) = 0, \  0 \le \alpha \le p, \ \alpha \ \text{odd} \bigr\},
\end{equation}
where
\begin{equation*}
\boldsymbol{\tau}_{p,1} \coloneqq
\begin{cases}
\,\bigl\{0, \frac{1/2}{n}, \frac{3/2}{n}, \ldots, \frac{n-1/2}{n}, 1\bigr\}, & p \text{ odd},\\[3pt]
\,\bigl\{0, \frac{1}{n}, \frac{2}{n}, \ldots, \frac{n-1}{n}, 1\bigr\}, & p \text{ even};
\end{cases}
\end{equation*}
\item mixed boundary conditions:
\begin{equation}\label{S_p_2}
\begin{aligned}
\mathbb{S}_{p,2} \coloneqq \bigl\{ s \in \mathbb{S}_{p,\boldsymbol{\tau}_{p,2}} : \partial^{\alpha_0} s(0) = \partial^{\alpha_1} s(1) = 0,\ & \ 0 \le \alpha_0,\alpha_1 \le p,\ \\
& \ \alpha_0 \ \text{even}, \alpha_1 \ \text{odd} \bigr\},
\end{aligned}
\end{equation}
where
\begin{equation*}
\boldsymbol{\tau}_{p,2} \coloneqq
\begin{cases}
\,\bigl\{0, \frac{2}{2n+1}, \frac{4}{2n+1}, \ldots, \frac{2n}{2n+1}, 1\bigr\}, & p \text{ odd},\\[3pt]
\,\bigl\{0, \frac{1}{2n+1}, \frac{3}{2n+1}, \ldots, \frac{2n-1}{2n+1}, 1\bigr\}, & p \text{ even}.
\end{cases}
\end{equation*}
\end{itemize}
When taking $\mathbb{V}_h=\mathbb{S}_{p,i}$, $i=0,1,2$ in the Galerkin approach, accurate approximations are obtained for all the first $n$ eigenvalues and eigenfunctions of the Laplace operator, with boundary conditions of type $0,1,2$, respectively \cite{Manni2022}. For this reason, such spaces are called \emph{outlier-free}.

\begin{remark} \label{rmk:reduced-spaces}
Similar subspaces $\overline{\mathbb{S}}_{p,i}$, $i = 0, 1$, introduced for uniform partitions $\boldsymbol{\tau}$ in \cite{Sogn2019,Takacs2016} and further analyzed in \cite[Section~5.2]{Sande2020}, were considered for outlier removal in \cite{Hiemstra2021} (and also \cite{Lamsahel2025}). Compared to the optimal spaces $\mathbb{S}_{p,i}$ , $i = 0, 1$, these subspaces can slightly differ in the partition and in the maximum order of vanishing derivatives at the boundary, depending on the parity of the degree $p$. However, we have that $\mathbb{S}_{p,0} = \overline{\mathbb{S}}_{p,0}$ for $p$ odd and that $\mathbb{S}_{p,1} = \overline{\mathbb{S}}_{p,1}$ for $p$ even.
\end{remark}

It is worth mentioning that the eigenvalues of $M^{-1}_h K_h$ do not depend on the choice of the basis in the considered discretization space. On the contrary, those of the mass matrix $M_h$ and of the stiffness matrix $K_h$ clearly  depend on the selected basis. In the next section, we present special bases for the spaces $\mathbb{S}_{p,i}$, $i=0,1,2$, such that the corresponding mass and stiffness matrices enjoy particularly nice spectral properties.

\subsection{Outlier-Free Spline Bases}
The effective use of the spaces $\mathbb{S}_{p,i}$ is clearly tied to the possibility of representing their elements in terms of basis functions with pleasing properties, such as local support and non-negativity. In this regard, the following bases were proposed in \cite{Divona2019,Manni2022} (see also \cite{FloaterSande2019,Lamsahel2025}):
\begin{itemize}
\item
for $n \geq p + 1$, the basis $\mathscr{E}_{p,0}$ for the space $\mathbb{S}_{p,0}$ is given by the $n$ functions
\begin{equation}\label{Base_Eliseo}
\mathcal{E}_{j,p,0}(x) \coloneqq
\begin{cases}
 -\mathcal{N}_{q-j,p}(y) + \mathcal{N}_{q+j,p}(y),
 & j \in \left\{ 1, \ldots, \left\lfloor \frac{p}{2} \right\rfloor \right\}, \\[3pt]
 \mathcal{N}_{q+j,p}(y),
 & j \in \left\{ \left\lfloor \frac{p}{2} \right\rfloor + 1, \ldots,
 n - \left\lfloor \frac{p}{2} \right\rfloor \right\}, \\[3pt]
 \mathcal{N}_{q+j,p}(y) - \mathcal{N}_{q + 2(n+1) - j,p}(y),
 & j \in \left\{ n + 1 - \left\lfloor \frac{p}{2} \right\rfloor, \ldots, n \right\},
\end{cases}
\end{equation}
where
$q \coloneqq \left\lfloor \frac{p}{2} \right\rfloor - p$,
$y(x) \coloneqq (n+1)x$ if $p$ odd,
$y(x) \coloneqq (n+1)x + \frac{1}{2}$ if $p$ even;
\item for $n \geq 2p - 2\lfloor \frac{p}{2} \rfloor + 1$, the basis $\mathscr{E}_{p,1}$ for the space $\mathbb{S}_{p,1}$ is given by the $n$ functions
\begin{equation}\label{Base_Eliseo_2}
\mathcal{E}_{j,p,1}(x) \coloneqq
\begin{cases}
\mathcal{N}_{-\left\lfloor \frac{p}{2} \right\rfloor - j,p}(y)
+ \mathcal{N}_{-\left\lfloor \frac{p}{2} \right\rfloor + j - 1,p}(y),
& j \in \left\{ 1, \ldots, -q \right\}, \\[3pt]
\mathcal{N}_{-\left\lfloor \frac{p}{2} \right\rfloor + j - 1,p}(y),
& j \in \left\{ -q + 1, \ldots, n + q \right\}, \\[3pt]
\mathcal{N}_{-\left\lfloor \frac{p}{2} \right\rfloor + j - 1,p}(y)
+ \mathcal{N}_{-\left\lfloor \frac{p}{2} \right\rfloor + 2n - j,p}(y),
& j \in \left\{ n + q + 1, \ldots, n \right\},
\end{cases}
\end{equation}
where
$q \coloneqq \left\lfloor \frac{p}{2} \right\rfloor - p$,
$y(x) \coloneqq nx + \frac{1}{2}$ if $p$ odd,
$y(x) \coloneqq nx$ if $p$ even;
\item for $n \geq p + 1$, the basis $\mathscr{E}_{p,2}$ for the space $\mathbb{S}_{p,2}$ is given by the $n$ functions
\begin{equation}\label{Base_Eliseo_3}
\mathcal{E}_{j,p,2}(x) \coloneqq \mathcal{N}_{q+j,p}(y) + \gamma_j \mathcal{N}_{q-k_j,p}(y),
\qquad j = 1, \ldots, n,
\end{equation}
where
\begin{equation*}
(\gamma_j, k_j) \coloneqq
\begin{cases}
(-1, j), & j \in \left\{ 1, \ldots, \left\lfloor \frac{p}{2} \right\rfloor \right\}, \\[3pt]
(0, \cdot), & j \in \left\{ \left\lfloor \frac{p}{2} \right\rfloor + 1, \ldots, n + q \right\}, \\[3pt]
(1, j - (2n + 1)), & j \in \left\{ n + q + 1, \ldots, n \right\},
\end{cases}
\end{equation*}
and
$q \coloneqq \left\lfloor \frac{p}{2} \right\rfloor - p$,
$y(x) \coloneqq \frac{2n + 1}{2}\, x$ if $p$ odd,
$y(x) \coloneqq \frac{2n + 1}{2}\, x + \frac{1}{2}$ if $p$ even.
\end{itemize}

\begin{remark} \label{rmk:reduced-bases}
Similar basis constructions, defined in terms of cardinal B-splines, were also developed for the subspaces mentioned in Remark~\ref{rmk:reduced-spaces}. Such a basis for the space $\overline{\mathbb{S}}_{p,1}$ was introduced in \cite{Takacs2016} and for the space $\overline{\mathbb{S}}_{p,0}$ in \cite{Manni2022} (see also \cite{Lamsahel2025}). An alternative basis construction was proposed in \cite{Hiemstra2021}.
\end{remark}

By construction, the basis functions in \eqref{Base_Eliseo}--\eqref{Base_Eliseo_3} inherit the good computational properties of cardinal B-splines (and so of B-splines). Furthermore, they possess some remarkable spectral properties as detailed in the following.

Given a basis $\mathscr{B}$ of one of the discretization spaces \eqref{S_p_0}--\eqref{S_p_2}, we denote by
$M_{\mathscr{B}}$ and $K_{\mathscr{B}}$ the corresponding mass and stiffness matrices.
Recall that $1\leq p \in \mathbb{N}$.
For $r=0,1$ and $\theta \in [0,\pi]$, let
\begin{equation}\label{simboli}
g^r_p(\theta)
\coloneqq (-1)^r \mathcal{N}^{(2r)}_{2p+1}(p+1)
+ 2(-1)^r \sum_{k=1}^{p}
\mathcal{N}^{(2r)}_{2p+1}(p+1-k)\cos(k\theta).
\end{equation}
It is known that $g^0_p(\theta)>0$ for $\theta \in[0,\pi]$, and $g^1_p(\theta)>0$ for $\theta \in(0,\pi]$; see \cite{Garoni2014}. Moreover, their ratio
\begin{equation}\label{monotonia}
e_p(\theta) \coloneqq \frac{g^1_p(\theta)}{g^0_p(\theta)}
\end{equation}
is monotone increasing on $[0,\pi]$; see \cite{Ekstrom2018}.
The following theorems were shown in \cite{Lamsahel2025} for the bases $\mathscr{E}_{p,i}$, $i = 0,1,2$, introduced above.

\begin{theorem}\label{U}
Let $\mathscr{E}_{p,0}$ be as in \eqref{Base_Eliseo}. Assume
$n \geq \max\!\left\{ p+1,\; p+\left\lfloor \frac{p}{2} \right\rfloor - 1 \right\}$.
Then, $M_{\mathscr{E}_{p,0}}$ and $K_{\mathscr{E}_{p,0}}$ are simultaneously diagonalizable by the orthogonal matrix $U$ consisting of the elements
\begin{equation*}
U_{i,j} \coloneqq \sqrt{\frac{2}{n+1}} \sin\!\left( \frac{ij\pi}{n+1} \right),
\end{equation*}
and their eigenvalues are given by
\begin{equation*}
\lambda_j \bigl( M_{\mathscr{E}_{p,0}} \bigr)
= (n + 1)^{-1} g_p^{0}\!\left( \frac{j\pi}{n + 1} \right),
\quad
\lambda_j \bigl( K_{\mathscr{E}_{p,0}} \bigr)
= (n + 1) g_p^{1}\!\left( \frac{j\pi}{n + 1} \right),
\end{equation*}
for $i,j=1,\ldots, n$.
\end{theorem}

\begin{theorem}\label{U_2}
Let $\mathscr{E}_{p,1}$ be as in \eqref{Base_Eliseo_2}. Assume
$n \geq \max\!\left\{ 2p-\left\lfloor \frac{p}{2} \right\rfloor,\;
2p-2\left\lfloor \frac{p}{2} \right\rfloor + 1 \right\}$.
Then, $M_{\mathscr{E}_{p,1}}$ and $K_{\mathscr{E}_{p,1}}$ are simultaneously diagonalizable by the orthogonal matrix $U$ consisting of the elements
\begin{equation*}
U_{i,j}
\coloneqq \sqrt{\frac{2}{n}}\, c_j
\cos\!\left( \frac{(j-1)\pi}{n}
\left(i - \frac{1}{2}\right) \right), \quad \text{ with } \quad c_j \coloneqq
\begin{cases}
\frac{1}{\sqrt{2}}, & j = 1, \\[3pt]
1, & j \geq 2,
\end{cases}
\end{equation*}
and their eigenvalues are given by
\begin{equation*}
\lambda_j \bigl( M_{\mathscr{E}_{p,1}} \bigr)
= n^{-1} g_p^{0}\!\left( \frac{(j-1)\pi}{n} \right),
\quad
\lambda_j \bigl( K_{\mathscr{E}_{p,1}} \bigr)
= n g_p^{1}\!\left( \frac{(j-1)\pi}{n} \right),
\end{equation*}
for $i,j=1,\ldots, n$.
\end{theorem}

\begin{theorem}\label{U_3}
Let $\mathscr{E}_{p,2}$ be as in \eqref{Base_Eliseo_3}. Assume
$n \geq \max\!\left\{ p+1,\; p + \left\lfloor \frac{p}{2} \right\rfloor \right\}$.
Then, $M_{\mathscr{E}_{p,2}}$ and $K_{\mathscr{E}_{p,2}}$ are simultaneously diagonalizable by the orthogonal matrix $U$ consisting of the elements
\begin{equation*}
U_{i,j}
\coloneqq \sqrt{\frac{4}{2n+1}}
\sin\!\left( \frac{i(2j-1)\pi}{2n+1} \right),
\end{equation*}
and their eigenvalues are given by
\begin{equation*}
\lambda_j \bigl( M_{\mathscr{E}_{p,2}} \bigr)
= \left( \frac{2n+1}{2} \right)^{-1}
g^0_p\!\left( \frac{(2j-1)\pi}{2n+1} \right),
\quad
\lambda_j \bigl( K_{\mathscr{E}_{p,2}} \bigr)
= \left( \frac{2n+1}{2} \right)
g^1_p\!\left( \frac{(2j-1)\pi}{2n+1} \right),
\end{equation*}
for $i,j=1,\ldots, n$.
\end{theorem}

Theorems~\ref{U} to \ref{U_3} and the monotonicity of the function \eqref{monotonia} directly lead to the following corollary.
\begin{corollary}\label{Corolla0}
For $i = 0,1,2$, the generalized eigenvalues of
$K_{\mathscr{E}_{p,i}}\mathbf{u} = \lambda M_{\mathscr{E}_{p,i}}\mathbf{u}$ are given by the ratios
$
{\lambda_j \bigl( K_{\mathscr{E}_{p,i}} \bigr)}/
{\lambda_j \bigl( M_{\mathscr{E}_{p,i}} \bigr)}$,
$ j =1,\dots,n$,
and they are all distinct.
\end{corollary}

\begin{remark}
The precise spectral knowledge of the matrices $M_{\mathscr{E}_{p,i}}$ and $K_{\mathscr{E}_{p,i}}$, $i = 0,1,2$, follows from the fact that they exhibit a Toeplitz-minus-Hankel or Toeplitz-plus-Hankel structure \cite{Lamsahel2025}.
Such structured matrices belong to certain $\tau$ matrix algebras, whose (spectral) study dates back to the works \cite{Bini1983,Bozzo1995,Difiore1995}. These structures have also been investigated more recently in \cite{Deng2021}.
For the discretization matrices related to standard (full) spline spaces, such structure is only available in few specific cases of low degree~\cite{Ekstrom2018}.
\end{remark}

The relevance of the above mentioned spectral properties of the bases in \eqref{Base_Eliseo}--\eqref{Base_Eliseo_3} is twofold, involving two different meanings of the notion \emph{outlier-free}.
\begin{enumerate}
\item For $i=0,1,2$, the matrices $M_{\mathscr{E}_{p,i}}$ and $K_{\mathscr{E}_{p,i}}$ admit a closed-form description of their spectra, i.e., all their eigenvalues are known samples, up to a scaling, of the functions $g_p^0$ and $g_p^1$ in \eqref{simboli}, which are the symbols of $M_{\mathscr{E}_{p,i}}$ and $K_{\mathscr{E}_{p,i}}$, respectively \cite{Garoni2014}. For this reason, the basis $\mathscr{E}_{p,i}$ is called \emph{outlier-free}. Note that the functions $g_p^0$ and $g_p^1$ are independent of the index $i$, i.e., of the type of boundary
conditions.
\item For $i=0,1,2$, the matrices $M_{\mathscr{E}_{p,i}}$ and $K_{\mathscr{E}_{p,i}}$ are simultaneously diagonalizable, i.e., the eigenvectors of the stiffness matrix all coincide with the eigenvectors of the mass matrix -- and so with those of $M^{-1}_{\mathscr{E}_{p,i}} K_{\mathscr{E}_{p,i}}$. This property ensures that the generalized eigenvalues of
$K_{\mathscr{E}_{p,i}}\mathbf{u} = \lambda M_{\mathscr{E}_{p,i}}\mathbf{u}$ agree with the ratios $
{\lambda_j \bigl( K_{\mathscr{E}_{p,i}} \bigr)}/
{\lambda_j \bigl( M_{\mathscr{E}_{p,i}} \bigr)}$. By taking into account the known explicit expressions of ${\lambda_j \bigl( M_{\mathscr{E}_{p,i}} \bigr)}$ and $ {\lambda_j \bigl( K_{\mathscr{E}_{p,i}} \bigr)}$, it implies that the corresponding space $\mathbb{S}_{p,i}$ is \emph{outlier-free} as discussed in \cite{Lamsahel2025}.
\end{enumerate}
This raises the question whether there are other bases of the space $\mathbb{S}_{p,i}$, $i=0,1,2$, enjoying the same (or very similar) spectral properties and still maintaining a B-spline-like structure (with properties such as local support and non-negativity). This will be investigated in the next section.

\section{Outlier-Free and B-Spline-Like Bases for Outlier-Free Spaces}\label{sec:4}
This section contains the main result of the paper, which is stated in Theorem~\ref{TEOREMA}. We show that any basis $\mathscr{B}_{p,i}$ of the space $\mathbb{S}_{p,i}$, $i=0,1,2$,
such that
\begin{enumerate}
\item the elements of $\mathscr{B}_{p,i}$  have the same support structure as those of $\mathscr{E}_{p,i}$,
\item the corresponding mass and stiffness matrices are simultaneously diagonalizable,
\end{enumerate}
can be obtained from the basis $\mathscr{E}_{p,i}$ by a change of basis identified by an orthogonal matrix $R_{\mathscr{B}_{p,i}}$. This implies (see Corollary~\ref{Basi_outlierfree})
that the mass and stiffness matrices with respect to the basis
$\mathscr{B}_{p,i}$ share the same spectrum as the corresponding  matrices
with respect to the basis $\mathscr{E}_{p,i}$, and thus they are outlier-free as well.

\subsection{Main Result}\label{caratterizzazione}
Before we are able to formulate our main result, we need a precise definition of the family of bases of $\mathbb{S}_{p,i}$ we are interested in. To this end, we start by recalling a well-known linear algebra result; see, e.g., \cite{HornJohnson2013}.
\begin{lemma}\label{Horn}
Let $A$ and $B$ be diagonalizable matrices in $\mathbb{R}^{n \times n}$. Then, $A$ and $B$ commute if and only if they are simultaneously diagonalizable.
\end{lemma}

We observe that for a basis $\mathscr{B}$, the corresponding mass matrix $M_{\mathscr{B}}$ and
stiffness matrix $K_{\mathscr{B}}$ are symmetric and thus diagonalizable by orthogonal transformations. This means that, according to Lemma~\ref{Horn}, they are simultaneously diagonalizable if and only if their commutator vanishes, namely
\begin{equation*}
[M_{\mathscr{B}}, K_{\mathscr{B}}]
\coloneqq M_{\mathscr{B}} K_{\mathscr{B}} - K_{\mathscr{B}} M_{\mathscr{B}} = 0.
\end{equation*}
In this perspective, we consider bases $\mathscr{B}_{p,i}$ of $\mathbb{S}_{p,i}$ satisfying
$[M_{\mathscr{B}_{p,i}}, K_{\mathscr{B}_{p,i}}] = 0$.

Furthermore, we want that the basis elements have the same support structure as those of $\mathscr{E}_{p,i}$.
More precisely, from \eqref{Base_Eliseo}--\eqref{Base_Eliseo_3}, we deduce the following technical lemma.
\begin{lemma}
For $i = 0,1,2$, let
$\mathscr{B}_{p,\boldsymbol{\tau}_{p,i}}$ be the open-knot B-spline basis of $\mathbb{S}_{p,\boldsymbol{\tau}_{p,i}}$; see \eqref{osaitu}. Then,
\begin{equation*}
\mathscr{E}_{p,i} = A_{\mathscr{E}_{p,i}}\,\mathscr{B}_{p,\boldsymbol{\tau}_{p,i}},
\end{equation*}
where $A_{\mathscr{E}_{p,i}}\in\mathbb{R}^{n \times n_i}$, with
$n_i = n + p + 1$ if $p$ is odd, for all $i = 0,1,2$, and $n_0 = n + p + 2$, $n_1 = n + p$ and $n_2 = n + p + 1$ if $p$ is even.
Moreover, the matrix has the following block structure:
\begin{equation}\label{A_E}
A_{\mathscr{E}_{p,i}}
=
\begin{bmatrix}
A^i_{11} & 0 & 0 \\ 
0 & I_{r_i \times r_i} & 0 \\ 
0 & 0 & A^i_{33}
\end{bmatrix},
\quad
A_{11}^i \in \mathbb{R}^{d_{1,i} \times d_{2,i}},
\quad
A_{33}^i \in \mathbb{R}^{d_{3,i} \times d_{4,i}},
\end{equation}
where for $p$ odd:
\begin{itemize}
\item $i = 0$: $d_{1,i} = d_{3,i} = \left\lfloor \frac{p}{2} \right\rfloor$ and $d_{2,i} = d_{4,i} = p$,
\item $i = 1$: $d_{1,i} = d_{3,i} = \left\lceil \frac{p}{2} \right\rceil$ and $d_{2,i} = d_{4,i} = p + 1$,
\item $i = 2$: $d_{1,i} = \left\lfloor \frac{p}{2} \right\rfloor$, $d_{3,i} = \left\lceil \frac{p}{2} \right\rceil$, and $d_{2,i} = p, d_{4,i} = p + 1,$
\end{itemize}
and for $p$ even:
\begin{itemize}
\item $i = 0$: $d_{1,i} = d_{3,i} = \left\lfloor \frac{p}{2} \right\rfloor$ and $d_{2,i} = d_{4,i} = p + 1$,
\item $i = 1$: $d_{1,i} = d_{3,i} = \left\lceil \frac{p}{2} \right\rceil$ and $d_{2,i} = d_{4,i} = p $,
\item $i = 2$: $d_{1,i} = \left\lfloor \frac{p}{2} \right\rfloor$, $d_{3,i} = \left\lceil \frac{p}{2} \right\rceil$, and $d_{2,i} = p + 1, d_{4,i} = p $.
\end{itemize}
The blocks $A^i_{11}$ and $A^i_{33}$ have full rank, for all $i = 0,1,2$.
\end{lemma}

The open-knot B-spline basis $\mathscr{B}_{p,\boldsymbol{\tau}_{p,i}}$  is the basis with the most compact support of its elements for the space $\mathbb{S}_{p,\boldsymbol{\tau}_{p,i}}$, which induces
a pronounced sparsity structure in the mass and stiffness matrices. Due to the block structure of the matrix $A_{\mathscr{E}_{p,i}}$, the basis $\mathscr{E}_{p,i}$ mimics these features and modifies only those elements of the open-knot B-spline basis that are affected by the additional boundary conditions characterizing the space $\mathbb{S}_{p,i}$; see \eqref{S_p_0}--\eqref{S_p_2}.
It is therefore natural to consider other bases of the space $\mathbb{S}_{p,i}$ that preserve the same sparsity and support structure, i.e., bases that can be expressed in terms of the open-knot B-spline basis through a matrix having the same block structure as the matrix $A_{\mathscr{E}_{p,i}}$ in \eqref{A_E}.

We are now ready to define the family of bases of $\mathbb{S}_{p,i}$ of interest.

\begin{definition}\label{DEFINIZIONE}
For $i = 0,1,2$, we denote by $\mathfrak{F}_{p,i}$ the family of bases of $\mathbb{S}_{p,i}$
such that
\begin{enumerate}
\item $\mathscr{B}_{p,i} = A_{\mathscr{B}_{p,i}} \mathscr{B}_{p,\boldsymbol{\tau}_{p,i}}$, where
$A_{\mathscr{B}_{p,i}}$ has the same block structure as $A_{\mathscr{E}_{p,i}}$ in \eqref{A_E},
\item $M_{\mathscr{B}_{p,i}} K_{\mathscr{B}_{p,i}} - K_{\mathscr{B}_{p,i}} M_{\mathscr{B}_{p,i}} = 0$.
\end{enumerate}
\end{definition}

Our main result relates every element of this family to the basis $\mathscr{E}_{p,i}$.
\begin{theorem}\label{TEOREMA}
For $i=0,1,2$, let
\begin{align*}
f_0(p) &\coloneqq \max \left\{p+1,\; 3 \left\lfloor \tfrac{p}{2} \right\rfloor \right\}, \\
f_1(p) &\coloneqq \max \left\{ 2p - \left\lfloor \tfrac{p}{2} \right\rfloor,\; 2p - 2\left\lfloor \tfrac{p}{2} \right\rfloor + 1 \right\}, \\
f_2(p) &\coloneqq \max \left\{ p+1,\; p + 2\left\lceil \tfrac{p}{2} \right\rceil \right\}.
\end{align*}
For $n \geq f_i(p)$,
let $\mathscr{B}_{p,i}$ be a basis of $\mathbb{S}_{p,i}$ belonging to the family $\mathfrak{F}_{p,i}$, and let
$R_{\mathscr{B}_{p,i}} \in \mathbb{R}^{n \times n}$
such that
\begin{equation*}
\mathscr{B}_{p,i} = R_{\mathscr{B}_{p,i}} \mathscr{E}_{p,i}.
\end{equation*}
Then, $R_{\mathscr{B}_{p,i}}$ is orthogonal.
\end{theorem}
The proof of the theorem is provided in Section~\ref{dimostrazione}.
We immediately deduce the following corollary.
\begin{corollary}\label{Basi_outlierfree}
For $i = 0,1,2$, let $n \geq f_i(p)$ and let $\mathscr{B}_{p,i}$ be a basis of $\mathbb{S}_{p,i}$ belonging to the family $\mathfrak{F}_{p,i}$. Then, the mass matrix $M_{\mathscr{B}_{p,i}}$ and the stiffness matrix $K_{\mathscr{B}_{p,i}}$ have the same spectra as the matrices $M_{\mathscr{E}_{p,i}}$ and $K_{\mathscr{E}_{p,i}}$, respectively. In particular, $\mathscr{B}_{p,i}$ is an outlier-free basis.
\end{corollary}
\begin{proof}
Let $R_{\mathscr{B}_{p,i}}$ be the change-of-basis matrix such that
$\mathscr{B}_{p,i} = R_{\mathscr{B}_{p,i}} \mathscr{E}_{p,i}$.
Then, from the definition of the mass and stiffness matrices in \eqref{stiff-mass}, we obtain
\begin{equation}\label{MK-conv}
M_{\mathscr{B}_{p,i}} = R_{\mathscr{B}_{p,i}} M_{\mathscr{E}_{p,i}} R_{\mathscr{B}_{p,i}}^{T}
\quad \text{and} \quad
K_{\mathscr{B}_{p,i}} = R_{\mathscr{B}_{p,i}} K_{\mathscr{E}_{p,i}} R_{\mathscr{B}_{p,i}}^{T}.
\end{equation}
Since $R_{\mathscr{B}_{p,i}}$ is orthogonal by Theorem~\ref{TEOREMA},
the matrix $M_{\mathscr{B}_{p,i}}$ has the same spectrum as $M_{\mathscr{E}_{p,i}}$, and the same holds for $K_{\mathscr{B}_{p,i}}$ and $K_{\mathscr{E}_{p,i}}$.
\end{proof}

Note that, vice versa, whenever the basis $\mathscr{B}_{p,i}$ is obtained from $\mathscr{E}_{p,i} $ by an orthogonal matrix, then \eqref{MK-conv} implies that $M_{\mathscr{B}_{p,i}} K_{\mathscr{B}_{p,i}} - K_{\mathscr{B}_{p,i}} M_{\mathscr{B}_{p,i}} = 0$.

\subsection{Proof of Theorem~\ref{TEOREMA}}\label{dimostrazione}
This section is devoted to the proof of Theorem~\ref{TEOREMA}, which makes use of Chebyshev polynomials, graph theory, and few other preliminary lemmas. For the sake of brevity, we will only report the proof for the case $\mathbb{S}_{p,0}$. The proof strategy for the cases $\mathbb{S}_{p,1}$ and $\mathbb{S}_{p,2}$ is similar, with minor technical changes. For further details, we refer to \cite{Ricci2025}.

We start with a series of lemmas.
In the first lemma, we show a trigonometric identity.
\begin{lemma}\label{trigo}
Let $m \in \mathbb{N}$ and $\alpha \in \mathbb{R}$ such that $\sin(\alpha)\ne 0$.
Then,
\begin{equation*}
 \sum_{k=1}^m \cos\bigl((2k-1)\alpha\bigr)
= \frac{\sin(2m\alpha)}{2\sin\alpha}.
\end{equation*}
\end{lemma}
\begin{proof}
We have
\begin{equation*}
S \coloneqq \sum_{k=1}^m e^{2\textrm{i}k\alpha} = e^{2\textrm{i}\alpha}\,\frac{1 - e^{2\textrm{i}m\alpha}}{1 - e^{2\textrm{i}\alpha}},
\end{equation*}
with $\textrm{i} \coloneqq \sqrt{-1}$,
because it is a finite geometric series with ratio $e^{2\textrm{i}\alpha}$.
By means of the relation
\begin{equation*}
1 - e^{2\textrm{i}m\alpha}
= e^{\textrm{i}m\alpha}\left(e^{-\textrm{i}m\alpha} - e^{\textrm{i}m\alpha}\right)
= -2\textrm{i}\,e^{\textrm{i}m\alpha}\sin(m\alpha),
\end{equation*}
we obtain
\begin{equation*}
S =
e^{2\textrm{i}\alpha}\,
\frac{-2\textrm{i}\,e^{\textrm{i}m\alpha}\sin(m\alpha)}
{-2\textrm{i}\,e^{\textrm{i}\alpha}\sin(\alpha)}
=
\frac{e^{\textrm{i}(m+1)\alpha}\sin(m\alpha)}{\sin(\alpha)}.
\end{equation*}
Then, denoting by $\Re(z)$ the real part of $z\in\mathbb{C}$, we arrive at
\begin{align*}
\sum_{k=1}^m \cos\bigl((2k-1)\alpha\bigr)&=\sum_{k=1}^m \Re(e^{\textrm{i}(2k-1)\alpha}) =\Re( e^{-\textrm{i}\alpha}S)
=\frac{\sin(m\alpha)}{\sin(\alpha)} \Re(e^{\textrm{i}m\alpha}) \\
&=\frac{\cos(m\alpha)\sin(m\alpha)}{\sin(\alpha)} = \frac{\sin(2m\alpha)}{2\sin\alpha},
\end{align*}
which completes the proof.
\end{proof}

Let $\operatorname{im}(A)$ and $\ker(A)$ denote the image and kernel of a matrix $A$.
We now look at a special matrix $C$, of interest later in the proof of Theorem~\ref{TEOREMA}, and state some of its properties.
\begin{lemma}\label{CC}
Consider the matrix $C\in  \mathbb{R}^{(n+1)\times n} $ consisting of the elements
\begin{equation*}
C_{i,j} \coloneqq \cos\!\left(\frac{i j \pi}{n+1}\right),
\quad
i=1,\dots,n+1,\quad j=1,\dots,n.
\end{equation*}
Then, we have
\begin{enumerate}
\item $C$ has rank $n$;
\item $\mathbf{1}_{\mathrm{even}} \in \operatorname{im}(C)$ and
      $\mathbf{1}_{\mathrm{odd}} \notin \operatorname{im}(C)$, where $\mathbf{1}_{\mathrm{even}}$ ($\mathbf{1}_{\mathrm{odd}}$) stands for the vector in $\mathbb{R}^n$ with entries equal to 1 for the even (odd) components and 0 otherwise.
\end{enumerate}
\end{lemma}
\begin{proof}
To prove the first statement, we define the points
\begin{equation}\label{x_i}
x_i \coloneqq \cos\!\left(\frac{i\pi}{n+1}\right),
\quad i=1,\ldots,n+1,
\end{equation}
and use the identity
\begin{equation*}
T_{k}(x_i)
= \cos\!\left(k\arccos(x_i)\right)
= \cos\!\left(\frac{ik\pi}{n+1}\right),
\quad i=1,\ldots,n+1,
\quad k=0,\ldots,n,
\end{equation*}
where $T_k$ stands for the Chebyshev polynomial of the first kind of degree $k$.
We observe that the matrix $C$ is composed of $n$ columns of the collocation matrix of the Chebyshev polynomials up to degree $n$ at the points \eqref{x_i}.
Since the points $x_i$ are distinct, there exists a unique interpolating
polynomial in $\mathbb{P}_n$, and since the polynomials
$\{T_0,\ldots,T_n\}$ form a basis of $\mathbb{P}_n$, it follows that the columns of the collocation matrix are linearly independent, and so $C$ has full column rank $n$.

To prove the second statement, recall that, in general,
\begin{equation}\label{CAM2}
\operatorname{im}(C) = \ker(C^{T})^{\perp}.
\end{equation}
Since $C$ has rank $n$, $\ker(C^{T})$ is one-dimensional.
We explicitly construct a generator of $\ker(C^{T})$ in the following. We distinguish between two cases depending on the parity of~$n$.
\begin{itemize}
\item $n$ even:
Let $\mathbf{v}\in\mathbb{R}^{n+1}$ be the vector consisting of the elements
\begin{equation*}
v_r \coloneqq
\begin{cases}
0, & \text{if $r$ is even}, \\
1, & \text{if $r$ is odd and $r\neq n+1$}, \\
\frac{1}{2}, & \text{if $r=n+1$}.
\end{cases}
\end{equation*}
Then, for all $r\in\{1,\ldots,n\}$,
\begin{equation*}
(C^{T}\mathbf{v})_r
= \sum_{k=1}^{n/2}
\cos\!\left(\frac{(2k-1) r \pi}{n+1}\right)
+  \frac{(-1)^r}{2},
\end{equation*}
and from Lemma~\ref{trigo}, we obtain
\begin{equation*}
\sum_{k=1}^{n/2}
\cos\!\left(\frac{(2k-1) r \pi}{n+1}\right)
= -\frac{(-1)^r}{2}.
\end{equation*}
Thus, we have $C^{T}\mathbf{v}=0$.
\item $n$ odd:
Consider the vector $\mathbf{1}_{\mathrm{odd}}$. Then, for all
$r\in\{1,\ldots,n\}$,
\begin{equation*}
(C^{T}\mathbf{1}_{\mathrm{odd}})_r
= \sum_{k=1}^{(n+1)/2}
\cos\!\left(\frac{(2k-1) r \pi}{n+1}\right)
= 0,
\end{equation*}
again by Lemma~\ref{trigo}.
Thus, we have $C^{T}\mathbf{1}_{\mathrm{odd}}=0$.
\end{itemize}
We conclude that the vectors $\mathbf{v}$ and $\mathbf{1}_{\mathrm{odd}}$ generate $\ker(C^{T})$ for $n$ even and $n$ odd, respectively.
Since
\begin{equation*}
\langle \mathbf{1}_{\mathrm{even}}, \mathbf{v} \rangle
= \langle \mathbf{1}_{\mathrm{even}}, \mathbf{1}_{\mathrm{odd}} \rangle
= 0,
\end{equation*}
it follows from~\eqref{CAM2} that
$\mathbf{1}_{\mathrm{even}} \in \operatorname{im}(C)$.
On the other hand,
\begin{equation*}
\langle \mathbf{1}_{\mathrm{odd}}, \mathbf{v} \rangle \neq 0,
\quad
\langle \mathbf{1}_{\mathrm{odd}}, \mathbf{1}_{\mathrm{odd}} \rangle \neq 0,
\end{equation*}
and therefore $\mathbf{1}_{\mathrm{odd}} \notin \operatorname{im}(C)$.
\end{proof}

For a given basis $\mathscr{B}_{p,0} \in \mathfrak{F}_{p,0}$, the next lemma relates the structure of the matrix $A_{\mathscr{B}_{p,0}}$ and the structure of the change-of-basis matrix $R_{\mathscr{B}_{p,0}}$.
\begin{lemma}\label{Fulcrodetutto}
Let $\mathscr{B}_{p,0} = A_{\mathscr{B}_{p,0}} \mathscr{B}_{p,\boldsymbol{\tau}_{p,0}} = R_{\mathscr{B}_{p,0}} \mathscr{E}_{p,0}$.
Then,
\begin{equation}\label{R_struttura}
R_{\mathscr{B}_{p,0}} =
\begin{bmatrix}
R_{11} & 0 & 0 \\
0 & I_{r\times r} & 0 \\
0 & 0 & R_{33}
\end{bmatrix},
\quad
R_{11}, R_{33} \in \mathbb{R}^{\left\lfloor \frac{p}{2} \right\rfloor \times \left\lfloor \frac{p}{2} \right\rfloor}.
\end{equation}
\end{lemma}
\begin{proof}
Since the structure of $A_{\mathscr{B}_{p,0}}$ and $A_{\mathscr{E}_{p,0}}$ is the same (see \eqref{A_E}), the conclusion is straightforward.
\end{proof}

We now build a matrix $H$ linked to the structure of the matrix $R_{\mathscr{B}_{p,0}}$ in \eqref{R_struttura} and determine its rank by exploiting graph theory. To this end, we define a set of index pairs, which identifies some zero entries of the first row and of the $\left\lfloor \frac{p}{2}\right\rfloor$-th column of $R_{\mathscr{B}_{p,0}}$.
\begin{lemma}\label{HH}
Let $n \geq f_0(p)$ and let $\mathcal{I}_p$ be the set of $n-1$ index pairs defined by
\begin{equation*}
\begin{aligned}
 &\{(i,j) : i = 1 \text{ and } p+1 \le j \le n \} \\
&\qquad \cup
\left\{(i,j) : \left\lfloor \tfrac{p}{2} \right\rfloor + 1 \le i \le
\left\lfloor \tfrac{p}{2} \right\rfloor + p-1 \text{ and }
j = \left\lfloor \tfrac{p}{2} \right\rfloor \right\},
\end{aligned}
\end{equation*}
if $p$ is odd, and by
\begin{equation*}
\{(i,j) : i = 1 \text{ and } p+2 \le j \le n \}
\cup
\left\{ (i,j) : \tfrac{p}{2} + 1 \le i \le \tfrac{p}{2} + p \text{ and }
j = \tfrac{p}{2} \right\},
\end{equation*}
if $p$ is even, where we assume that a set is empty if the upper bound of the indices is less than the lower bound.
Moreover, let
$H \in \mathbb{R}^{(n-1) \times (n+1)}$ be the matrix whose $k$-th row is specified by the $k$-th index pair $(i,j)$ of the set $\mathcal{I}_p$ as
\begin{equation*}
H_{k,\bullet} \coloneqq \mathbf{e}_{|i-j|} - \mathbf{e}_{t(i+j)}
\quad k =1,\ldots,n-1,
\end{equation*}
where $\mathbf{e}_r$ denotes the $r$-th vector of the canonical basis and
\begin{equation*}
t(i+j) \coloneqq
\begin{cases}
i + j, & \text{if } i + j \le n, \\
2(n + 1) - (i + j), & \text{if } i + j > n.
\end{cases}
\end{equation*}
Then, we have
$\operatorname{rank}(H) = n - 1$.
\end{lemma}
\begin{proof}
We start the proof by observing that $H$ can be regarded as the transpose of an incidence matrix of a directed graph (see Section~\ref{grafi}).
Therefore, it suffices to show that this graph has two connected components and the rank of $H$ follows from Theorem~\ref{Grafi1}. In the following, we describe a constructive procedure to build the two connected components.
\begin{enumerate}
\item We first build two main sequences of index pairs $(i,j)$ belonging to $\mathcal{I}_p \cup (1, n+1) \cup (1, n+2)$:
\begin{align}
\label{psp}
\left(\left\lfloor \tfrac{p}{2} \right\rfloor + 1, \left\lfloor \tfrac{p}{2} \right\rfloor\right)
& \rightsquigarrow \left(1, p + 1 \right) \rightsquigarrow \left(1, p+3 \right)
\rightsquigarrow \left(1, p+5 \right) \rightsquigarrow  \cdots \rightsquigarrow (1, \ell),\\
\label{ssp}
\left(\left\lfloor \tfrac{p}{2} \right\rfloor + 2, \left\lfloor \tfrac{p}{2} \right\rfloor\right)
& \rightsquigarrow (1,p+2) \rightsquigarrow (1,p+4) \rightsquigarrow (1,p+6)
\rightsquigarrow  \cdots \rightsquigarrow (1,\ell),
\end{align}
if $p$ is odd, and
\begin{align}
\label{pspp}
\left( \tfrac{p}{2} + 1, \tfrac{p}{2} \right)
& \rightsquigarrow (1,p+2) \rightsquigarrow (1,p+4) \rightsquigarrow (1,p+6)
 \rightsquigarrow \cdots \rightsquigarrow (1,\ell),\\
\label{sspp}
\left( \tfrac{p}{2} + 2, \tfrac{p}{2} \right)
& \rightsquigarrow \left(1, p + 3 \right) \rightsquigarrow \left(1, p+5 \right)
\rightsquigarrow \left(1, p+7 \right)
\rightsquigarrow \cdots \rightsquigarrow (1, \ell),
\end{align}
if $p$ is even, where $\ell \in \{n+1,n+2\}$ depending on the parity of $n$ and $p$.
For example, if $p$ is odd and $n$ is even, since $n > p$, then $r = n - p$ is odd and $(1, p+r) = (1, n)$, which means that the sequence \eqref{psp} has $\ell = n+2$ and \eqref{ssp} has $\ell = n+1$ in this case. We need to use the same reasoning in the other cases.
Thus, in total, we count $n-p+4$ index pairs in \eqref{psp}--\eqref{ssp} and $n-p+3$ index pairs in \eqref{pspp}--\eqref{sspp}. To each of these index pairs $(i,j)$, we assign the vertex $v_{|i-j|}$. For each $|i-j|<\ell-1$, we consider an edge from $v_{|i-j|}$ (tail) to $v_{t(i+j)}$ (head).
\item We then connect the remaining index pairs to the previous sequences as follows:
\begin{alignat}{2}
\label{aus1d}
\left(\left\lfloor \tfrac{p}{2} \right\rfloor + 2k+1, \left\lfloor \tfrac{p}{2} \right\rfloor\right)
 &\rightsquigarrow \left(1, t(p + 2k) + 1 \right),
 \quad &k &= 1, 2,\ldots, \tfrac{p-3}{2}, \\
\label{aus2d}
\left(\left\lfloor \tfrac{p}{2} \right\rfloor + 2k, \left\lfloor \tfrac{p}{2} \right\rfloor\right)
 &\rightsquigarrow \left(1, t(p + 2k -1) + 1 \right),
 \quad &k &= 2, 3, \ldots, \tfrac{p-1}{2},
\end{alignat}
if $p$ is odd, and
\begin{alignat}{2}
\label{aus1p}
\left( \tfrac{p}{2} + 2k+1, \tfrac{p}{2} \right)
 &\rightsquigarrow \left(1, t(p + 2k + 1) + 1 \right),
 \quad &k &= 1, 2, \ldots, \tfrac{p-2}{2}, \\
\label{aus2p}
\left( \tfrac{p}{2} + 2k, \tfrac{p}{2} \right)
 &\rightsquigarrow \left(1, t(p + 2k) + 1 \right),
 \quad &k &= 2, 3, \ldots, \tfrac{p}{2},
\end{alignat}
if $p$ is even.
Thus, in total, we have $p-3$ additional index pairs in \eqref{aus1d}--\eqref{aus2d} and $p-2$ additional index pairs in \eqref{aus1p}--\eqref{aus2p}. To each of these index pairs $(i,j)$, we assign the vertex $v_{|i-j|}$, and we consider an edge from $v_{|i-j|}$ (tail) to $v_{t(i+j)}$ (head).
Note that, since $n\geq f_0(p)$, \eqref{aus1d} ensures $p\leq t(i+j)\leq \ell-1$ and \eqref{aus2d} gives $p+1\leq t(i+j)\leq \ell-1$, for $p$ odd. Similarly, \eqref{aus1p} ensures $p+1\leq t(i+j)\leq \ell-1$ and
\eqref{aus2p} gives $p+2\leq t(i+j)\leq \ell-1$, for $p$ even.
\end{enumerate}
After putting together the two steps above, we obtain a directed graph with $n+1$ vertices and $n-1$ edges, consisting of two connected components. For example, assuming $n\gg p$, it can be visualized as
\begin{center}
\begin{tikzpicture}[
scale=1.1,
every node/.style={font=\small},
vertex/.style={circle, draw, fill=gray!30, inner sep=1.6pt, minimum size=6pt},
edge/.style={-{Latex[length=2mm]}, thin}]
\def\dy{1.0}
\def\dx{1.2}
\node[vertex,label=above:$v_1$]         (v1)   at (0,0) {};
\node[vertex,label=above:$v_p$]         (vp)   at (1*\dx,0) {};
\node[vertex,label=above:$v_{p+2}$]     (vp1)  at (2*\dx,0) {};
\node[vertex,label=above:$v_{p+4}$]     (vp2)  at (3*\dx,0) {};
\node[vertex,label=above:$v_{p+6}$]     (vp3)  at (4*\dx,0) {};
\node[label=above:$\cdots$]             (dots) at (5*\dx,0) {};
\node[vertex,label=above:$v_{\ell-1}$]  (vl1)  at (6*\dx,0) {};
\draw[edge] (v1) -- (vp);
\draw[edge] (vp) -- (vp1);
\draw[edge] (vp1) -- (vp2);
\draw[edge] (vp2) -- (vp3);
\draw[edge,dashed] (vp3) -- (vl1);
\node[vertex,label=below:$v_3$] (v3) at (2*\dx, -\dy) {};
\node[vertex,label=below:$v_5$] (v5) at (3*\dx, -\dy) {};
\node[vertex,label=below:$v_7$] (v7) at (4*\dx, -\dy) {};
\draw[edge] (v3) -- (vp1);
\draw[edge] (v5) -- (vp2);
\draw[edge] (v7) -- (vp3);
\end{tikzpicture}

\begin{tikzpicture}[
scale=1.1,
every node/.style={font=\small},
vertex/.style={circle, draw, fill=gray!30, inner sep=1.6pt, minimum size=6pt},
edge/.style={-{Latex[length=2mm]}, thin}]
\def\dy{1.0}
\def\dx{1.2}
\node[vertex,label=above:$v_2$]         (v2)   at (0,0) {};
\node[vertex,label=above:$v_{p+1}$]     (vp1)  at (1*\dx,0) {};
\node[vertex,label=above:$v_{p+3}$]     (vp2)  at (2*\dx,0) {};
\node[vertex,label=above:$v_{p+5}$]     (vp3)  at (3*\dx,0) {};
\node[vertex,label=above:$v_{p+7}$]     (vp4)  at (4*\dx,0) {};
\node[label=above:$\cdots$]             (dots) at (5*\dx,0) {};
\node[vertex,label=above:$v_{\ell-1}$]  (vl1)  at (6*\dx,0) {};
\draw[edge] (v2) -- (vp1);
\draw[edge] (vp1) -- (vp2);
\draw[edge] (vp2) -- (vp3);
\draw[edge] (vp3) -- (vp4);
\draw[edge,dashed] (vp4) -- (vl1);
\node[vertex,label=below:$v_4$]  (v4)  at (2*\dx, -\dy) {};
\node[vertex,label=below:$v_6$]  (v6)  at (3*\dx, -\dy) {};
\node[vertex,label=below:$v_8$]  (v8)  at (4*\dx, -\dy) {};
\draw[edge] (v4) -- (vp2);
\draw[edge] (v6) -- (vp3);
\draw[edge] (v8) -- (vp4);
\end{tikzpicture}
\end{center}
if $p$ is odd, and
\begin{center}
\begin{tikzpicture}[
scale=1.1,
every node/.style={font=\small},
vertex/.style={circle, draw, fill=gray!30, inner sep=1.6pt, minimum size=6pt},
edge/.style={-{Latex[length=2mm]}, thin}]
\def\dy{1.0}
\def\dx{1.2}
\node[vertex,label=above:$v_1$]         (v1)   at (4*\dx,0) {};
\node[vertex,label=above:$v_{p+1}$]     (vp1)  at (5*\dx,0) {};
\node[vertex,label=above:$v_{p+3}$]     (vp2)  at (6*\dx,0) {};
\node[vertex,label=above:$v_{p+5}$]     (vp3)  at (7*\dx,0) {};
\node[vertex,label=above:$v_{p+7}$]     (vp4)  at (8*\dx,0) {};
\node[label=above:$\cdots$]             (dots) at (9*\dx,0) {};
\node[vertex,label=above:$v_{\ell-1}$]  (vl1)  at (10*\dx,0) {};
\draw[edge] (v1) -- (vp1);
\draw[edge] (vp1) -- (vp2);
\draw[edge] (vp2) -- (vp3);
\draw[edge] (vp3) -- (vp4);
\draw[edge,dashed] (vp4) -- (vl1);
\node[vertex,label=below:$v_3$] (v3) at (6*\dx, -\dy) {};
\node[vertex,label=below:$v_5$] (v5) at (7*\dx, -\dy) {};
\node[vertex,label=below:$v_7$] (v7) at (8*\dx, -\dy) {};
\draw[edge] (v3) -- (vp2);
\draw[edge] (v5) -- (vp3);
\draw[edge] (v7) -- (vp4);
\end{tikzpicture}

\begin{tikzpicture}[
scale=1.1,
every node/.style={font=\small},
vertex/.style={circle, draw, fill=gray!30, inner sep=1.6pt, minimum size=6pt},
edge/.style={-{Latex[length=2mm]}, thin}]
\def\dy{1.0}
\def\dx{1.2}
\node[vertex,label=above:$v_2$]         (v2)   at (4*\dx,0) {};
\node[vertex,label=above:$v_{p+2}$]     (vp1)  at (5*\dx,0) {};
\node[vertex,label=above:$v_{p+4}$]     (vp2)  at (6*\dx,0) {};
\node[vertex,label=above:$v_{p+6}$]     (vp3)  at (7*\dx,0) {};
\node[vertex,label=above:$v_{p+8}$]     (vp4)  at (8*\dx,0) {};
\node[label=above:$\cdots$]             (dots) at (9*\dx,0) {};
\node[vertex,label=above:$v_{\ell-1}$]  (vl1)  at (10*\dx,0) {};
\draw[edge] (v2) -- (vp1);
\draw[edge] (vp1) -- (vp2);
\draw[edge] (vp2) -- (vp3);
\draw[edge] (vp3) -- (vp4);
\draw[edge,dashed] (vp4) -- (vl1);
\node[vertex,label=below:$v_4$]  (v4)  at (6*\dx, -\dy) {};
\node[vertex,label=below:$v_6$]  (v6)  at (7*\dx, -\dy) {};
\node[vertex,label=below:$v_8$]  (v8)  at (8*\dx, -\dy) {};
\draw[edge] (v4) -- (vp2);
\draw[edge] (v6) -- (vp3);
\draw[edge] (v8) -- (vp4);
\end{tikzpicture}
\end{center}
if $p$ is even.
It is easy to check that $H$ coincides with the transpose of the incidence matrix of this graph.
Therefore, Theorem~\ref{Grafi1} implies that
$\operatorname{rank}(H) = (n+1) - 2 = n - 1$.
\end{proof}

Finally, our last lemma  determines the rank of the matrix  $HC$, where $H$  and $C$ have been described above.
\begin{lemma}\label{IMPORTANTE}
Let $C$ be as in Lemma~\ref{CC} and let $H$ be as in Lemma~\ref{HH}. Then,
$\operatorname{rank}(HC) = n - 1$.
\end{lemma}
\begin{proof}
In general, the following identity holds:
\begin{equation*}
\operatorname{rank}(HC)
= \operatorname{rank}(C) - \dim\!\left(\operatorname{im}(C) \cap \ker(H)\right).
\end{equation*}
By item 1 of Lemma~\ref{CC} we have $\operatorname{rank}(C) = n$. Moreover, by the definition of $H$ and Lemma~\ref{HH},
\begin{equation*}
\operatorname{span}\!\left(\mathbf{1}_{\mathrm{odd}},\, \mathbf{1}_{\mathrm{even}}\right)
= \ker(H).
\end{equation*}
On the other hand, by item 2 of Lemma~\ref{CC}, we have
$\mathbf{1}_{\mathrm{even}} \in \operatorname{im}(C)$ and
$\mathbf{1}_{\mathrm{odd}} \notin \operatorname{im}(C)$, so
\begin{equation*}
\dim\!\left(\operatorname{im}(C) \cap \ker(H)\right) = 1.
\end{equation*}
Therefore,
$\operatorname{rank}(HC) = n - 1$.
\end{proof}

We have now at our disposal all the ingredients for the proof of Theorem~\ref{TEOREMA}.
\begin{proof}[Proof of Theorem~\ref{TEOREMA}]
From $\mathscr{B}_{p,0} = R_{\mathscr{B}_{p,0}} \mathscr{E}_{p,0}$, it follows that
\begin{equation*}
M_{\mathscr{B}_{p,0}} = R_{\mathscr{B}_{p,0}} M_{\mathscr{E}_{p,0}} R_{\mathscr{B}_{p,0}}^{T}
\quad \text{and} \quad
K_{\mathscr{B}_{p,0}} = R_{\mathscr{B}_{p,0}} K_{\mathscr{E}_{p,0}} R_{\mathscr{B}_{p,0}}^{T}.
\end{equation*}
Therefore, since $\mathscr{B}_{p,0}\in \mathfrak{F}_{p,0}$ and by Definition~\ref{DEFINIZIONE}, we have
\begin{equation*}
R_{\mathscr{B}_{p,0}} M_{\mathscr{E}_{p,0}} W K_{\mathscr{E}_{p,0}} R_{\mathscr{B}_{p,0}}^{T}
-
R_{\mathscr{B}_{p,0}} K_{\mathscr{E}_{p,0}} W M_{\mathscr{E}_{p,0}} R_{\mathscr{B}_{p,0}}^{T}
= 0,
\end{equation*}
with $W \coloneqq R_{\mathscr{B}_{p,0}}^{T} R_{\mathscr{B}_{p,0}}$, which is equivalent to
\begin{equation}\label{eqfond}
M_{\mathscr{E}_{p,0}} W K_{\mathscr{E}_{p,0}}
-
K_{\mathscr{E}_{p,0}} W M_{\mathscr{E}_{p,0}}
= 0
\end{equation}
by the invertibility of $R_{\mathscr{B}_{p,0}}$. From Theorem~\ref{U}, we know that there
exists an orthogonal and symmetric matrix $U$ that diagonalizes
$M_{\mathscr{E}_{p,0}}$ and $K_{\mathscr{E}_{p,0}}$. Thus, \eqref{eqfond} can be rewritten as
\begin{equation}\label{W}
\Lambda_{M_{\mathscr{E}_{p,0}}}\, U W U\, \Lambda_{K_{\mathscr{E}_{p,0}}}
-
\Lambda_{K_{\mathscr{E}_{p,0}}}\, U W U\, \Lambda_{M_{\mathscr{E}_{p,0}}}
= 0,
\end{equation}
where we used the invertibility of $U$, and where
$\Lambda_{M_{\mathscr{E}_{p,0}}}$ and $\Lambda_{K_{\mathscr{E}_{p,0}}}$ denote the diagonal eigenvalue matrices of $M_{\mathscr{E}_{p,0}}$ and $K_{\mathscr{E}_{p,0}}$, respectively.
Setting $\hat{W} \coloneqq U W U$, we can read \eqref{W} entrywise:
\begin{align*}
&\bigl(\Lambda_{M_{\mathscr{E}_{p,0}}} \hat{W} \Lambda_{K_{\mathscr{E}_{p,0}}}
- \Lambda_{K_{\mathscr{E}_{p,0}}} \hat{W} \Lambda_{M_{\mathscr{E}_{p,0}}}\bigr)_{i,j} \\
&\qquad=
\hat{W}_{i,j}\bigl(
\lambda_i(M_{\mathscr{E}_{p,0}}) \lambda_j(K_{\mathscr{E}_{p,0}})
-
\lambda_i(K_{\mathscr{E}_{p,0}}) \lambda_j(M_{\mathscr{E}_{p,0}})
\bigr)\\
&\qquad=
\lambda_j(M_{\mathscr{E}_{p,0}}) \lambda_i(M_{\mathscr{E}_{p,0}})\, \hat{W}_{i,j}
\left(
\frac{\lambda_j(K_{\mathscr{E}_{p,0}})}{\lambda_j(M_{\mathscr{E}_{p,0}})}
-
\frac{\lambda_i(K_{\mathscr{E}_{p,0}})}{\lambda_i(M_{\mathscr{E}_{p,0}})}
\right)
= 0.
\end{align*}
Corollary~\ref{Corolla0} ensures that
$\frac{\lambda_k(K_{\mathscr{E}_{p,0}})}{\lambda_k(M_{\mathscr{E}_{p,0}})}$ are all distinct for $k = 1,\ldots,n$.
Thus,
$\hat{W}_{i,j} = 0$ for all $i,j$ such that $i \neq j$,
or in other words, $\hat{W}$ is a diagonal matrix. This means that
\begin{equation*}
W = R_{\mathscr{B}_{p,0}}^{T} R_{\mathscr{B}_{p,0}}
= U \operatorname{diag}(\mathbf{w})\, U,
\end{equation*}
where $\mathbf{w} \coloneqq (w_1,\ldots,w_n)^T \in \mathbb{R}^n$ is a vector with positive entries.

For each index pair $(i,j) \in \{1,\ldots,n\} \times \{1,\ldots,n\}$ such that $i \neq j$, the corresponding elements of $W$ satisfy
\begin{equation}\label{H}
W_{i,j}
= \sum_{k=1}^{n} w_k U_{i,k} U_{j,k}
= \sum_{k=1}^{n} w_k
\sin\!\left(\frac{ik\pi}{n+1}\right)
\sin\!\left(\frac{jk\pi}{n+1}\right)
= 0.
\end{equation}
Using the trigonometric identity
$\sin(\alpha)\sin(\beta)
= \frac{1}{2}\bigl(\cos(\alpha-\beta) - \cos(\alpha+\beta)\bigr)$,
we can rewrite \eqref{H} as
\begin{equation}\label{Wij}
W_{i,j}
=
\sum_{k=1}^{n}\frac{w_k}{2}
\left(
\cos\!\left(\frac{(i-j)k\pi}{n+1}\right)
-
\cos\!\left(\frac{(i+j)k\pi}{n+1}\right)
\right)
= 0.
\end{equation}
Recall now the following properties of the cosine function,
\begin{equation*}
\cos(-x) = \cos(x),
\quad\text{and}\quad
\cos\!\left(\left(2(a+1)-b\right)\frac{k\pi}{a+1}\right)
=
\cos\!\left(\frac{b k\pi}{a+1}\right),
\end{equation*}
with $x \in \mathbb{R}$ and $a,b,k \in \mathbb{N}$.
We select $m$ equations of the form \eqref{H}, and to each of the index pairs $(i,j)$, we assign an integer $k=1,\ldots,m$. Then, we can rewrite them compactly using \eqref{Wij} and the cosine matrix $C$ defined in Lemma~\ref{CC} as follows:
\begin{equation*}
Z C \mathbf{w} = 0,
\quad
Z_{k,\bullet} \coloneqq \mathbf{e}_{|i-j|} - \mathbf{e}_{t(i+j)},
\quad  k  = 1,\ldots,m,
\end{equation*}
where $\mathbf{e}_r$ denotes the $r$-th vector of the canonical basis and
\begin{equation*}
t(i+j) \coloneqq
\begin{cases}
i + j, & \text{if } i + j \le n, \\
2(n + 1) - (i + j), & \text{if } i + j > n.
\end{cases}
\end{equation*}

Since $\mathscr{B}_{p,0} \in \mathfrak{F}_{p,0}$, Lemma~\ref{Fulcrodetutto} implies that $R_{\mathscr{B}_{p,0}}$ is block diagonal. The same holds for the matrix $W$, and the index pairs in Lemma~\ref{HH} identify zero entries of $W$ according to such a structure. The set $\mathcal{I}_p$ in Lemma~\ref{HH} consists of $n-1$ index pairs $(i,j)$. By
imposing the corresponding $n-1$ equations \eqref{Wij}, the matrix $Z$
resulting from this choice of indices agrees with the matrix $H$ in
Lemma~\ref{HH}, which has $\operatorname{rank}(H) = n-1$.
Lemma~\ref{IMPORTANTE} ensures that
\begin{equation*}
\operatorname{rank}(HC) = n - 1.
\end{equation*}
Moreover, by orthogonality of the columns of $U$, it is clear that the specific choice $\mathbf{w} =\mathbf{1}$ satisfies \eqref{H}, which implies that $\mathbf{1} \in \ker(HC)$. Since $HC \in \mathbb{R}^{(n-1) \times n}$, we deduce
\begin{equation*}
\ker(HC) = \operatorname{span}(\mathbf{1}).
\end{equation*}
It follows that
$W = c I$.
Thus,
$R_{\mathscr{B}_{p,0}} = \sqrt{c}Q$
for some $c \in \mathbb{R}$ with $c > 0$ and for some orthogonal matrix $Q$. Since
$R_{\mathscr{B}_{p,0}}$ has the identity as its central block, see \eqref{R_struttura},
also the matrix $W$ has the identity as its central block, and therefore
$c = 1$.
We conclude that $R_{\mathscr{B}_{p,0}}$ is orthogonal.
\end{proof}

\section{Numerical Procedure}\label{sec:5}
In this section, we briefly present a numerical procedure for finding B-spline-like and outlier-free bases for the optimal spaces $\mathbb{S}_{p,i}$, without assuming any knowledge about the bases $\mathscr{E}_{p,i}$, $i = 0,1,2$.
This is motivated by the prospect of constructing similar bases -- if any --  for optimal spaces for the isogeometric discretization of biharmonic and polyharmonic eigenvalue problems \cite{Manni2023}, where no direct explicit constructions analogous to the bases $\mathscr{E}_{p,i}$ are available.
 
We focus on the case $\mathbb{S}_{p,0}$ with $p$ odd. The other cases can be addressed in a similar manner.
\begin{enumerate}
\item Let $A_{\mathscr{B}_{p,0}}$ be a candidate representation matrix for the new basis $\mathscr{B}_{p,0}$ in terms of the open-knot B-spline basis $\mathscr{B}_{p,\boldsymbol{\tau}_{p,0}}$, with the following block structure:
\begin{equation}\label{A_B}
A_{\mathscr{B}_{p,0}} = \begin{bmatrix}
A_{11} & 0 & 0 \\
0 & I_{r\times r} & 0 \\
0 & 0 & A_{33}
\end{bmatrix},
\quad
A_{11}, A_{33} \in
\mathbb{R}^{\left\lfloor \frac{p}{2} \right\rfloor \times p}.
\end{equation}
\item To ensure that $\mathscr{B}_{p,0}$ is a basis of the space $\mathbb{S}_{p,0}$, the basis elements must satisfy the homogeneous boundary conditions specified in \eqref{S_p_0}. Compute the even derivatives of the open-knot B-splines at the end points and set up the kernel of these conditions. Then, based on this kernel, determine the matrices $A_{11}$ and $A_{33}$ in terms of a set of parameters $X$, whose number only depends on the degree $p$.
\item Construct the mass and stiffness matrices $M_{\mathscr{B}_{p,0}}$ and $K_{\mathscr{B}_{p,0}}$ in terms of the mass and stiffness matrices of the open-knot B-spline basis $\mathscr{B}_{p,\boldsymbol{\tau}_{p,0}}$, using the representation matrix $A_{\mathscr{B}_{p,0}}$.
Then, minimize (numerically) the Frobenius norm of their commutator $\|M_{\mathscr{B}_{p,0}} K_{\mathscr{B}_{p,0}} - K_{\mathscr{B}_{p,0}} M_{\mathscr{B}_{p,0}}\|_{F}$ with respect to the above parameters $X$, imposing (linear) constraints on them that ensure the entries of $A_{11}$ and $A_{33}$ to be non-negative.
\item Build $A_{\mathscr{B}_{p,0}^*}$ using the computed minimized parameter values $X^*$ and obtain the new basis
$\mathscr{B}_{p,0}^* = A_{\mathscr{B}_{p,0}^*}\,\mathscr{B}_{p,\boldsymbol{\tau}_{p,0}}$.
\end{enumerate}
The solution space of the minimization problem in step~3 is non-empty and the optimal solution is non-unique. The obtained solution will depend on the chosen optimization method and can be possibly steered through initialization feed or additional constraints.
This numerical procedure finds justification in the previously presented theoretical result (Section~\ref{caratterizzazione}). Indeed, from steps~1 to~3, we deduce that the basis $\mathscr{B}_{p,0}^*$ belongs to the family $\mathfrak{F}_{p,0}$ specified in Definition~\ref{DEFINIZIONE}. From Corollary~\ref{Basi_outlierfree}, we conclude that the obtained basis is outlier-free. Furthermore, the non-negativity constraints in step~3 ensure that the resulting basis functions are non-negative, being non-negative combinations of (open-knot) B-splines.

We implemented the above numerical procedure and tested it out for the case $p=5$ and $n=7$. We obtained a basis $\mathscr{B}_{p,0}^*$ identified by the representation matrix $A_{\mathscr{B}_{p,0}^*}$ as in \eqref{A_B}, where
\begin{equation*}
A_{11}^* = \begin{bmatrix}
\ 0  & \ 0.1675 & \ 0.5024 & \ 0.8142 & \ 0.0859\ \\
\ 0  & \ 0.0023 & \ 0.0069 & \ 0.1305 & \ 0.9963\ \\
\end{bmatrix},
\end{equation*}
and $A_{33}^*$ has the same columns but in reversed order.
We remark that the matrix $A_{11}^0$ in \eqref{A_E}, corresponding to the basis $\mathscr{E}_{p,0}$, is given by
\begin{equation*}
A_{11}^0 = \begin{bmatrix}
\ 0 & \ 1/6  & \ 1/2  & \ 4/5 & \ 0\ \\
\ 0 & \ 1/60 & \ 1/20 & \ 1/5 & \ 1\ \\
\end{bmatrix},
\end{equation*}
and it holds
\begin{equation*}
A_{11}^* = R_{11}\, A_{11}^0,
\quad
R_{11} = \begin{bmatrix}
\  0.9963 & \ 0.0859\ \\
\ -0.0859 & \ 0.9963\ \\
\end{bmatrix}.
\end{equation*}
The matrix $R_{11}$ is clearly a rotation matrix (with angle $-0.086$). This is in agreement with Theorem~\ref{TEOREMA}, taking into account the matrix structure \eqref{R_struttura}. The corresponding bases $\mathscr{E}_{p,0}$ and $\mathscr{B}_{p,0}^*$ are depicted in Figure~\ref{FIGURA}. Since $r=3$ in \eqref{A_B}, the three central functions in both bases are standard (cardinal) B-splines. We close this numerical example by observing that any rotation matrix with an angle $\phi$ such that $0 \leq -\phi \leq \arctan(1/10)$ would result in a matrix $A_{11}^*$ with non-negative entries.

\begin{figure}[t!]
\centering
\begin{subfigure}{0.48\textwidth}
\includegraphics[width=\textwidth]{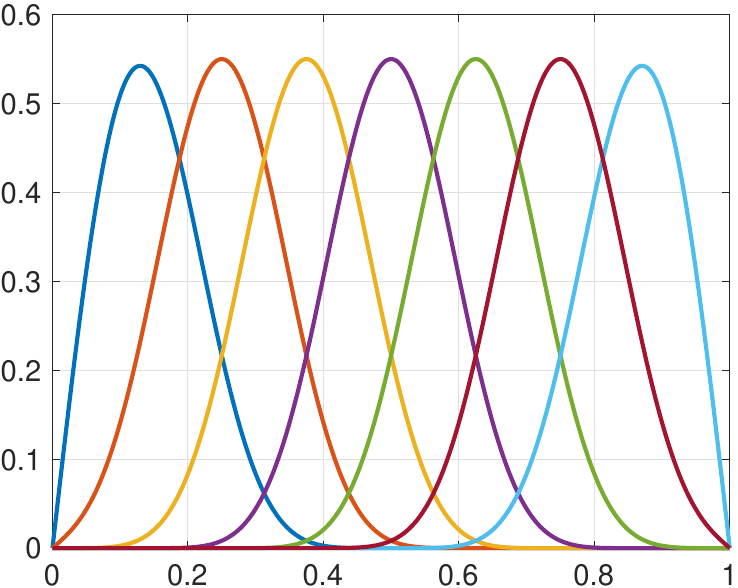}
\caption{the basis $\mathscr{E}_{p,0}$}
\end{subfigure}
\hfill
\begin{subfigure}{0.48\textwidth}
\includegraphics[width=\textwidth]{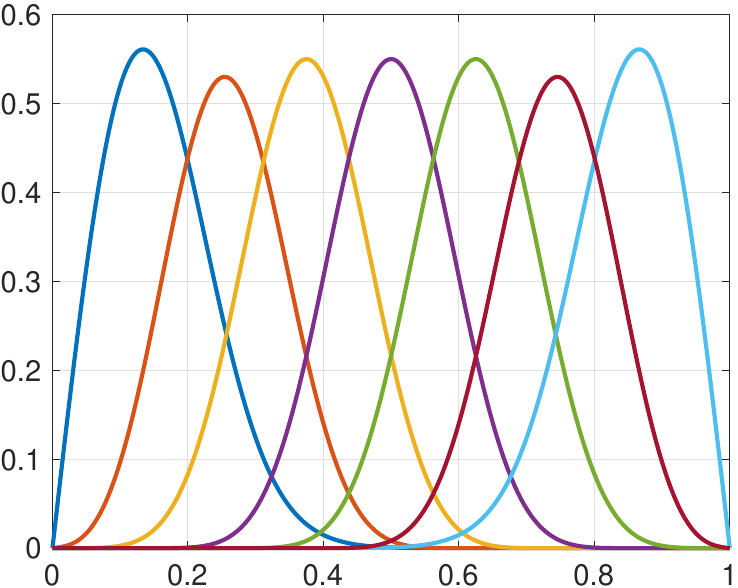}
\caption{an outlier-free basis $\mathscr{B}_{p,0}^*$}
\end{subfigure}
\caption{Two B-spline-like and outlier-free bases for the space $\mathbb{S}_{p,0}$, with $p=5$ and $n=7$: (a) the basis $\mathscr{E}_{p,0}$ defined in \eqref{Base_Eliseo}, and (b) a numerically obtained basis $\mathscr{B}_{p,0}^*$ as described in Section~\ref{sec:5}}\label{FIGURA}
\end{figure}

\section{Conclusion}\label{sec:6}
The use of optimal spline spaces in the isogeometric Galerkin method for the eigenvalue problem associated with the univariate Laplace operator -- subject to any classical boundary conditions: Dirichlet, Neumann, and mixed -- is a powerful approximation tool, which is theoretically proven to be outlier-free \cite{Lamsahel2025,Manni2022}.
While the absence of outliers is an intrinsic property of the considered discretization spaces $\mathbb{S}_{p,i}$, $i=0,1,2$, stemming from their optimality, the choice of the basis is of crucial importance from a practical point of view.

We have characterized the bases of these optimal spaces that enjoy a B-spline-like support structure and whose mass and stiffness matrices are simultaneously diagonalizable.
It turns out that such bases are orthogonally equivalent to the outlier-free bases $\mathscr{E}_{p,i}$ proposed in \cite{Divona2019,Manni2022}, and therefore share their important spectral properties derived in \cite{Lamsahel2025}. As a consequence, all the selected bases are also outlier-free.

The obtained results can be reasonably extended to the spaces proposed in \cite{Hiemstra2021} (see also \cite{Sogn2019,Takacs2016}), which are closely related to the optimal spline spaces considered here and still provide outlier-free discretizations for the eigenvalue problem associated with the univariate Laplace operator \cite{Lamsahel2025}.

Beyond their relevance for the Laplace operator, the obtained results and the related numerical procedure suggest a possible extension to the construction of similar bases for optimal spaces for the isogeometric discretization of biharmonic and polyharmonic eigenvalue problems \cite{Manni2023}.
Indeed, while a family of optimal spline spaces has been identified for the approximation of the eigenvalues of a polyharmonic operator of any order \cite{Manni2023}, the construction of suitable bases for such spaces remains an open question. Even in the biharmonic case, a direct attempt to construct a basis for the corresponding optimal spaces that reproduces the elegant expression of the basis $\mathscr{E}_{p,i}$ in terms of cardinal B-splines fails.
From this perspective, the results of this paper suggest to seek such a basis numerically by imposing a B-spline-like structure and minimizing a suitable norm of the commutator of the corresponding mass and stiffness matrices, in order to identify bases for which these matrices are simultaneously diagonalizable. This is a promising direction for future research.

\begin{acknowledgement}
C.~Manni and H.~Speleers are members of the research group GNCS (Gruppo Nazionale per il Calcolo Scientifico) of INdAM (Istituto Nazionale di Alta Matematica).
They have been supported by the MUR Excellence Department Project MatMod@TOV (CUP E83C23000330006) awarded to the Department of Mathematics of the University of Rome Tor Vergata, by the Department of Mathematics of the University of Rome Tor Vergata through the Project METRO (CUP  E83C25000630005), by a Project of Relevant
National Interest (PRIN) under the National Recovery and Resilience Plan (PNRR) funded by the European Union -- Next Generation EU (CUP E53D23017910001), and by the Italian Research Center on High Performance Computing, Big Data and Quantum Computing (CUP E83C22003230001).
\end{acknowledgement}


\begin{thebibliography}{99}
\bibitem{Bini1983}
Bini, D., Capovani, M.:
Spectral and computational properties of band symmetric Toeplitz matrices.
{Linear Algebra and its Applications} \textbf{52}, 99--126 (1983)

\bibitem{Boffi2010}
Boffi, D.:
Finite element approximation of eigenvalue problems.
{Acta Numerica} \textbf{19}, 1--120 (2010)

\bibitem{Bozzo1995}
Bozzo, E., Di Fiore, C.:
On the use of certain matrix algebras associated with discrete trigonometric transforms in matrix displacement decomposition.
{SIAM Journal on Matrix Analysis and Applications} \textbf{16}, 312--326 (1995)

\bibitem{Cottrell2006}
Cottrell, J.A., Reali, A., Bazilevs, Y., Hughes, T.J.R.:
Isogeometric analysis of structural vibrations.
{Computer Methods in Applied Mechanics and Engineering} \textbf{195}, 5257--5296 (2006)

\bibitem{Deng2021}
Deng, Q.:
Analytical solutions to some generalized and polynomial eigenvalue problems.
{Special Matrices} \textbf{9}, 240--256 (2021)

\bibitem{Difiore1995}
Di Fiore, C., Zellini, P.:
Matrix decompositions using displacement rank and classes of commutative matrix algebras.
{Linear Algebra and its Applications} \textbf{229}, 49--99 (1995)

\bibitem{Divona2019}
Di Vona, E.:
{Kolmogorov $n$-width e spazi spline ottimi}.
Tesi di Laurea Magistrale, Universit\`a degli Studi di Roma ``Tor Vergata'' (2019)

\bibitem{Ekstrom2018}
Ekstr{\"o}m, S.-E., Furci, I., Garoni, C., Manni, C., Serra-Capizzano, S., Speleers, H.:
Are the eigenvalues of the B-spline isogeometric analysis approximation of $-\Delta u = \lambda u$ known in almost closed form?.
{Numerical Linear Algebra with Applications} \textbf{25}, e2198 (2018)

\bibitem{FloaterSande2019}
Floater, M.S., Sande, E.:
Optimal spline spaces for $L^2$ $n$-width problems with boundary conditions.
{Constructive Approximation} \textbf{50}, 1--18 (2019)

\bibitem{Garoni2014}
Garoni, C., Manni, C., Pelosi, F., Serra-Capizzano, S., Speleers, H.:
On the spectrum of stiffness matrices arising from isogeometric	analysis.
{Numerische Mathematik} \textbf{127}, 751--799 (2014)

\bibitem{GodsilRoyle2001}
Godsil, C., Royle, G.:
{Algebraic Graph Theory}.
Springer (2001)

\bibitem{Hiemstra2021}
Hiemstra, R.R., Hughes, T.J.R., Reali, A., Schillinger, D.:
Removal of spurious outlier frequencies and modes from isogeometric discretizations of second- and fourth-order problems in one, two, and three dimensions.
{Computer Methods in Applied Mechanics and Engineering} \textbf{387}, 114115 (2021)

\bibitem{HornJohnson2013}
Horn, R.A., Johnson, C.R.:
{Matrix Analysis}, 2nd Ed.
Cambridge University Press (2013)

\bibitem{Lamsahel2025}
Lamsahel, N., Manni, C., Ratnani, A., Serra-Capizzano, S., Speleers, H.:
Outlier-free isogeometric discretizations for Laplace eigenvalue problems: closed-form eigenvalue and eigenvector expressions.
{Numerische Mathematik} \textbf{157}, 1397--1448 (2025)

\bibitem{Lyche2018}
Lyche, T., Manni, C., Speleers, H.:
Foundations of spline theory: B-splines, spline approximation, and hierarchical refinement.
In: Lyche, T., Manni, C., Speleers, H. (eds.) Splines and PDEs: From Approximation Theory to Numerical Linear Algebra. In: Lecture Notes in Mathematics, vol. 2219, pp. 1--76. Springer (2018)

\bibitem{Manni2022}
Manni, C., Sande, E., Speleers, H.:
Application of optimal spline subspaces for the removal of spurious outliers in isogeometric discretizations.
{Computer Methods in Applied Mechanics and Engineering} \textbf{389}, 114260 (2022)

\bibitem{Manni2023}
Manni, C., Sande, E., Speleers, H.:
Outlier-free spline spaces for isogeometric discretizations of biharmonic and polyharmonic eigenvalue problems.
{Computer Methods in Applied Mechanics and Engineering} \textbf{417}, 116314 (2023)

\bibitem{Ricci2025}
Ricci, D.:
{Basi outlier-free per spazi spline ottimi per l'operatore di Laplace}.
Tesi di Laurea Magistrale, Universit\`a degli Studi di Roma ``Tor Vergata'' (2025)

\bibitem{Sande2020}
Sande, E., Manni, C., Speleers, H.:
Explicit error estimates for spline approximation of arbitrary smoothness in isogeometric analysis.
{Numerische Mathematik} \textbf{144}, 889--929 (2020)

\bibitem{Sogn2019}
Sogn, J., Takacs, S.:
Robust multigrid solvers for the biharmonic problem in isogeometric analysis.
{Computers \& Mathematics with Applications} \textbf{77}, 105--124 (2019)

\bibitem{Takacs2016}
Takacs, S., Takacs, T.:
Approximation error estimates and inverse inequalities for {B}-splines of maximum smoothness.
{Mathematical Models and Methods in Applied Sciences} \textbf{26}, 1411--1445 (2016)

\end{thebibliography}
\end{document}